\documentclass[letterpaper,12pt]{article}
\usepackage{amsmath, amsthm, amssymb, graphicx,mathrsfs, polynom}

\theoremstyle{definition}
\newtheorem{thm}{Theorem}[section]

\newtheorem{lem}[thm]{Lemma}
\newtheorem{pro}[thm]{Proposition}

\newtheorem{dfc}[thm]{Definition-Construction}
\newtheorem{assumptions}[thm]{Assumptions}

\newtheorem{df}[thm]{Definition}
\newtheorem{rem}[thm]{Remark}
\newtheorem{exa}[thm]{Example}

\newcommand{\mB}{\mathcal{B}}
\newcommand{\mD}{\mathcal{D}}

\newcommand{\mA}{\mathcal{A}}

\newcommand{\mO}{\mathcal{O}}
\newcommand{\mK}{\mathcal{K}}

\newcommand{\mV}{\mathcal{V}}

\newcommand{\mF}{\mathcal{F}}

\newcommand{\mP}{\mathcal{P}}

\newcommand{\mT}{\mathcal{T}}

\newcommand{\mHom}{\mathcal{H}om}

\DeclareSymbolFont{AMSb}{U}{msb}{m}{n}
\DeclareMathSymbol{\boldk}{\mathord} {AMSb}{"7C}

\newcommand{\git}{/ \hspace{-3pt} / \hspace{-1pt}}

\newenvironment{lineindent}[1]
     {\begin{list}{}
             {\setlength{\leftmargin}{#1}}
             \item[]
     }
     {\end{list}}

\begin{document}
\title{\addtocounter{footnote}{1} On Localization for Quantum Hamiltonian Reductions in Arbitrary Characteristic}
\author{Theodore J. Stadnik, Jr. \footnote{This material is based upon work supported by the National Science 
Foundation under Grant No. 0932078 000, while the author was in 
residence at the Mathematical Science Research Institute (MSRI) in 
Berkeley, California, during the spring of 2013.} \\
	MSRI \\}
\date{}

\maketitle
\begin{abstract}
    \centering
    \begin{minipage}{0.6\textwidth}
For quantum Hamiltonian reductions in arbitrary characteristics, it is known that derived localization holds if and only if the algebra of global sections has finite global dimension.   In this paper we provide an alternative characterization of when derived localization holds:  Derived localization holds if and only if it holds for an explicit finite set of (quantized) line bundles.  As an application, we prove a new result that there are integral weights for which localization holds on in the positive characteristic hypertoric case for $p$ larger than an explicit bound. We also discuss how derived localization is a consequence of a finite number of Morita equivalences.
    \end{minipage}
\end{abstract}

\tableofcontents
\section{Introduction} 
The Beilinson-Bernstein localization theorem developed a relationship between the representation theory of $U(\mathfrak{g})$ and the geometry of the flag variety over the complex numbers \cite{BB}. In general, a localization theorem for an algebra $A$ is a pair of a space $X$ and a sheaf of algebras $\mA$ on $X$ such that (derived) global sections gives an equivalence between the category of $\mA$-modules and $A$-modules.  From the perspective of geometric representation theory, good localization theorems will provide interesting pairs $(X, \mA)$ so that the geometry of $X$ provides extra information about the structure of $\mA$-modules.  Since the introduction of the Beilinson-Bernstein theorem there have been many more examples of other algebras exhibiting localization such as Cherednik algebras \cite{KR} and hypertoric varieties \cite{MVdB}.  It has also been shown that localization occurs in positive characteristic in the Beilinson-Bernstein setting \cite{BMR} and Cherednik algebra settings \cite{BFG}.  There also has been significant progress in understanding localization in general settings \cite{BPW} \cite{MN}.\\

The aforementioned localization theorems are examples of localization theorems for quantum Hamiltonian reductions (QHRs).   For the purposes of this introduction, a QHR will be any (quantized) space $\mathfrak{X}_{\lambda}$ arising from the following process.  Let $Y$ be a (quantized) smooth symplectic variety equipped with an action of a group $G$ and an equivariant map $\mu: Y \rightarrow \mathfrak{g}^*.$   Fix a character $\delta \in X^*(G)$ and $\lambda \in g^*$ and define $\mathfrak{X}_{\lambda} = \mu^{-1}(\lambda)\git_{\delta} G$ to be the (quantized) GIT quotient.   It is a deformation of the commutative space $\mu^{-1}(0)\git_{\delta} G$.  It was first observed by \cite{BMR} in characteristic $p> 0$ and then more generally by \cite{MN} that derived localization in this context is equivalent to the algebra of global sections of the quantized structure sheaf having finite global dimension.  \\

In some instances it is rather straight-forward to quantify those $\lambda \in \mathfrak{g}^*$  for which global dimension of the algebra of global sections is finite; such as in the case of $U(\mathfrak{g})_{\lambda}$ in any characteristic.  Unfortunately, it is not always the case.  For hypertoric enveloping algebras in characteristic $0$, \cite{MVdB} proved the finiteness of global dimension of this algebra by proving localization holds.  This paper has two main goals:
\begin{enumerate}
\item Replace the condition of finite global dimensionality of the global sections by a condition which requires only a finite number of computations.
\item Show that the finite number of conditions found in $1$ allows us to quantify the size of the primes $p$ for which hypertoric localization holds in characteristic $p$.
\end{enumerate}

Before explaining these theorems in more technical nomenclature, we will first consider a more precise set-up.  Suppose $G$ is a connected reductive group acting on a smooth variety $X$.  The $G$-action on $X$ induces a $G$-action on the crystalline sheaf of differential operators $\mD_X$.   Differentiation of this action induces a noncommutative comoment map $\mu^*: U(\mathfrak{g}) \rightarrow \mD_X$. For $\delta \in X^*(G)$ and $\lambda \in \mathfrak{g}^*$ denote by $\mA_{\delta, \lambda}$ the $G$-invariants of the algebra $\mD_X$ (micro)localized at the $\delta$ semi-stable points of $T^*X$ modulo the relation induced by $\mu^*(\theta)-\lambda(\theta) = 0$ for all $\theta \in \mathfrak{g}$.  \\

\textbf{Theorem \ref{mainthm}.}  Under the mild assumptions of \ref{assumptions} the following are equivalent.
\begin{enumerate}
\item Global sections induces an equivalence of triangulated categories
$$D^b(Mod_{l.f.g}(\mA_{\lambda})) \cong D^b(Mod_{f.g.}(U_{\lambda})).$$
\item Enough of the noncommutative  ample line bundles right localize, that is
$$ \mathbb{R}\Gamma(\mA_{\lambda}(\delta^a))\otimes_{U_{\lambda}}^{\mathbb{L}} \mA_{\lambda} \cong \mA_{\lambda}(\delta^a)$$ for all $a \in S$ where $S$ is an explicit finite set of positive integers depending only on the action of $G$ on $X$.
\item The noncommutative ample line bundles $\{\mA_{\lambda}(\delta^a)| a \in S \}$ are in the Karoubian subcategory of $D^b(Mod^{right}_{l.f.g.}(\mA_{\lambda}))$ generated by $\mA_{\lambda}$.
\end{enumerate}

This theorem is then used to prove our main application.\\

\textbf{Theorem \ref{mainapp}.}  If $X = \mathbb{A}^n$ and $G \subset \mathbb{G}_m^n$ a subtorus then there is an explicit bound $M$ determined only by the map of groups $X_*(\mathbb{G}^{n}_m) \rightarrow X_*(\mathbb{G}^{n}_m/G)$ such that whenever $p > M$ there exists $\delta \in X^*(G)$ and $\lambda \in \mathfrak{g}^*(\mathbb{F}_p)$ such that
$$D^b(Mod_{l.f.g}(\mA_{\lambda})) \cong D^b(Mod_{f.g.}(U_{\lambda})).$$
Moreover, an upper bound for the size of $M$ is easy to compute and the set of such integral $\Lambda$ can be expressed from combinatorial data.\\

The main strategy \ref{mainapp} is to use the size of $p$ to force large but finite number of Morita equivalences.  In characteristic $0$ (underived) localization arises as a consequence of an infinite number of Morita equivalences (see \cite[Sec 5]{BPW}).  In our context, it is almost never the case in characteristic $p > 0$ that there are an infinite number of Morita equivalences.  Since the Morita theory approach in \ref{enoughmorita} applies to all QHRs (not just the toric case), it is the author's hope that this technique will be useful in a settings where the finiteness of the global dimension of the algebra of global sections is unknown.\\

The structure of this paper is as follows.  In section $2$ we outline basic information about notation, (micro)localization, and quantum Hamiltonian reductions.
The main theorem is \ref{mainthm} and its main ingredient is \ref{mainlem}.  Equivalent condition $2$ in the main theorem can be interpreted as a quantized version of Borel-Weil-Bott.  In section $3$, we discuss derived localization for quantum Hamiltonian reductions in a very general and characteristic free manner.  Section $4$ contains our examples.  The main examples of this paper are showing that derived localization is a consequence of an finite and explicit number of Morita equivalences; showing that there are integral parameters for which the hypertoric enveloping algebra localizes in characteristic $p >0$; and a short discussion on some results of \cite{BMR}. \\

\textbf{Acknowledgements.} The author would like to thank Tom Nevins for helpful discussions.

\section{Notation}
In these notes, $X$ will be a smooth variety over an algebraically closed field $\boldk$ of \textbf{arbitrary} characteristic equipped with the action of a connected reductive group $G$.  $\mD_X$ will denote the sheaf of crystalline differential operators on $X$.  It is the sheaf of algebras generated by the structure sheaf $\mO_X$ and the tangent sheaf $\mT_X$ subject to the usual commutation and Lie bracket relations.  The reader should consult \cite[1.2]{BMR} for more details. \\

For any algebra $A$ the category $Mod_{f.g.}(A)$ is the category of finitely generated $A$-modules.  For a sheaf of $\mO_Y$-algebras $\mA$ on a scheme $Y$, $Mod_{l.f.g.}(\mA)$ denotes the category of locally finitely generated $\mA$ modules which are quasi-coherent as $\mO_Y$-modules. $Hom(-,-)$ is the set of (global) module homomorphisms while $\mHom(-,-)$ denotes the sheaf of module homomorphisms.

\subsection{Moment maps}
The action of $G$ induces a noncommutative comoment map
$$\mu^*: \mathfrak{g} \rightarrow \mD_X(X)$$
via differentiating the left regular action of $G$ on functions on $X$.  By the universal property of enveloping algebras we can extend $\mu^*$ to a morphism $U(\mathfrak{g}) \rightarrow \mD_X(X)$ and to a morphism $\underline{U(\mathfrak{g})} \rightarrow \mD_X$ where $\underline{U(\mathfrak{g})}$ denotes the constant sheaf with sections $U(\mathfrak{g})$. The algebra $U(\mathfrak{g})$ and the sheaf $\mD_X$ are naturally equipped with PBW filtrations and the morphism $\underline{U(\mathfrak{g})} \rightarrow \mD_X$ is filtered.  Taking the relative spectrum of associated graded sheaves of algebras gives a map
$$\mu: T^*X \rightarrow \mathfrak{g}^*$$ 
which we will refer to as \textit{the moment map}.\footnote{Classically the moment map is defined by a list of properties which will be satisfied by more than one map if $G$ is not semi-simple.}\\

\subsection{Invariants}
For a character $\chi \in X^*(G)$ and a representation $V$ of $G$, $V^{\chi}$ will denote the subspace of $V$ consisting of all $v$ such that $G$ acts on $v$ by the character $\chi$.  When $\chi=1$ is the trivial character then we write $V^{G}$ to denote the $G$ invariants instead of $V^1$.  If $\mV$ is a $G$-equivariant sheaf on $X$ then $\mV^{\chi}$ and $\mV^{G}$ will denote the sheaves of $\chi$ semi-invariant sections and $G$-invariant sections respectively.
\\

\subsection{GIT quotients}\label{gitquot}
For any scheme $Y$ equipped with a $G$ action and $\delta \in X^*(G)$ we will denote the GIT quotient  by
$$Y \git_{\delta} G.$$
In the terminology of \cite{MFK}, $Y \git_{\delta} G$ is the uniform categorical quotient of the points  of $Y$ which are semi-stable for the trivial line bundle equipped with the linearization determined by $\delta$.
We will need the following facts from this construction:
\begin{enumerate}
\item $Y \git_{\delta} G$ is covered by affine open sets $U_{s \neq 0}$ where $s \in \mO_Y^{\delta^n}$ for some $n > 0$.
\item For any $U_{s \neq 0}$ from $1$ we have $\mO_{Y \git_{\delta} G}(U_{s \neq 0}) = \mO_Y(\{y| s(y) \neq 0\})^G$.
\item For the covering $\mathcal{U}=\{U_{s \neq 0}\}$ of $Y \git_{\delta} G$ from $1$, the set $Y^{\delta-uns} = \{ y | s(y) = 0 \text{ for all } U_{s \neq 0} \in \mathcal{U}\}$ is a closed set called the unstable locus.  Its complement is denoted $Y^{\delta-ss}$ and is the open set of semi-stable points.
\end{enumerate}

\subsection{Non-commutative localization}\label{NCloc}
The sheaf $\mO_{T^*X}$ is naturally graded as it is the associated graded sheaf of rings of the sheaf $\mD_X$ equipped with the PBW order filtration.  For any homogeneous element $s \in \mO_{T^*X}(T^*X)$ define $S \subset \mO_{T^*X}(T^*X)$ to be the smallest multiplicative set containing $S$ and $S_{sat} = \{ P \in \mD_X(X) | \sigma(P) \in S  \}$ where $\sigma$ denotes taking the principal symbol.\\

In \cite[2.1]{Li} it was shown that $S_{sat}$ is an Ore set.  Abusing notation we write $\mD(U_{s \neq 0})$ to indicate the localization of $\mD_X(X)$ at $S_{sat}$.  $\mD_X(V_{s \neq 0})$ is equipped with a filtration such that the associated graded ring is $\mO_{T^*X}(\{(x,\xi) | s(x,\xi) \neq 0\})$.  Loosely speaking, the collection of localizations $\{\mD_X(V_{s \neq 0})^G\}$ for $\{U_{s \neq 0}\}$ a covering of $T^*X \git_{\delta} G$ serves as a deformation of the space $T^*X \git_{\delta} G$ and is constructed through a quantized GIT quotient.  We will now make this more explicit but first we make the following notational convention.\\

\textbf{Convention.} If $s_1,...,s_l \in \mO_{T^*X}$ are homogeneous elements we will say a $\mD_X$-module $M$ vanishes on $\{(x,\xi) | s_i((x,\xi))=0 \text{ for all } i\}$ if $M \otimes_{\mD_X}\mD_X(U_{s_i \neq 0}) = 0$ for all $i$.  We will say a sequence of $\mD_X$-modules is exact on $\{(x,\xi) | s_i((x,\xi))=0 \text{ for all } i\}$ if its cohomology sheaves vanish on $\{(x,\xi) | s_i((x,\xi))=0 \text{ for all } i\}$.

\subsection{Quantized GIT Quotients}
 The natural action of $G$ on vector fields of $X$ induces a natural $G$-equivariant structure on $\mD_X$.  If $\delta \in X^*(G)$ and $\lambda \in \mathfrak{g}^*$ we obtain a sheaf of algebras
$$\mA_{\delta,\lambda}(U_{s \neq 0}) = \left(\mD_X(U_{s \neq 0})/\{\mu^*(\Theta)-\lambda(\Theta)\}_{\Theta \in \mathfrak{g}}\mD_X(U_{s \neq 0}) \right)^G$$ where $U_{s \neq 0}$ are the affine open sets covering $T^*X \git_{\delta} G$.  
The algebra $\mA_{\delta,\lambda}$ is a quantized version of the structure sheaf of the variety $\mu^{-1}(0) \git_{\delta} G$.\\

To see this is a ring, first consider that the $G$ action on $\mD_X/\{\mu^*(\Theta)-\lambda(\Theta)\}_{\Theta \in \mathfrak{g}}\mD_X$ differentiates to the action
$$\theta \cdot \overline{P} = \overline{[ \mu^*(\theta), P]}$$
where $\overline{Q}$ denotes passage from $\mD_X$  to the quotient and $[-,-]$ is the commutator bracket.  It follows that if $\overline{P_1}, \overline{P_2} \in \mA_{\delta,\lambda}$ then multiplication $\overline{P_1}\overline{P_2} = \overline{P_1P_2}$ is well-defined.

We  also have the corresponding global object,
$$U_{\lambda} = \Gamma(\mA_{\delta,\lambda}) = \left( \mD_X(X)/\{\mu^*(\Theta)-\lambda(\Theta)\}_{\Theta \in \mathfrak{g}}\mD_X(X) \right)^G,$$

called the enveloping algebra.  The algebra $U_{\lambda}$ does not depend on $\delta$.\\

\subsection{Simplification of localization when $char\boldk = p > 0$}
If $char \boldk = p > 0$ then there is a natural identification of the center with $\iota: \mO_{T^*X} \rightarrow Z(\mD_X)$  as discussed  in \cite[Sec 1]{BMR}. For a homogeneous section  $s \in \mO_{T^*X}(X)$ the ring $\mD_X[\iota(s)^{-1}]$ has the PBW order filtration and its associated graded ring is $\mO_{T^*X}(\{(x,\xi) | s(x,\xi) =0\})$. Thus it is enough to consider $\mD_X[\iota(s)^{-1}]$ instead of  $\mD_X(U_{s \neq 0})$ when $char \boldk =p>0$.  Abusively, in this case we redefine the symbol $\mD_X(U_{s \neq 0})$ as the ring $\mD_X[\iota(s)^{-1}]$.\\

This allows us to explicitly express $\mD_X(U_{s \neq 0})$ as a direct limit of finitely generated submodules when $X$ is affine.  Namely,
$$\mD_X(U_{s \neq 0}) \cong \varinjlim_l \mD_X(X)\iota(s)^{-l}$$
as right $\mD_X(X)$-modules.  Moreover, if $\iota(s)$ has weight $\gamma \in X^*(G)$ and $M$ is a $G$-equivariant $\mD_X$-module then
$$M(U_{s \neq 0})^{G}=\left(M \otimes_{\mD_X} \mD_X(U_{s \neq 0})\right)^G \cong \varinjlim_l M(X)^{\gamma^l}$$
as right $\mD_X(X)$-modules.  This fact will play an important role in our main theorem.\\

\subsection{Twisted modules}
For a character $\chi \in X^*(G)$ we will denote the $\chi$-twisted \textbf{right} $\mA_{\delta,\lambda}$-module (resp. $U_{\lambda}$-module) by
$$ \mA_{\delta,\lambda}(\chi)(U_{s \neq 0}) = \left(\mD_X(U_{s \neq 0})/ \{\mu^*(\Theta)-\lambda(\Theta)-d\chi(\Theta)\}_{\Theta \in \mathfrak{g}}\mD_X(U_{s \neq 0})\right)^{\chi}$$
$$(\text{resp. }U_{\lambda}(\chi) = \left(\mD_X(X)/ \{\mu^*(\Theta)-\lambda(\Theta)-d\chi(\Theta)\}_{\Theta \in \mathfrak{g}}\mD_X(X) \right)^{\chi}.$$

Note this module also carries a left action of $\mA_{\delta,\lambda+d\chi}$ (resp. $U_{\lambda+d\chi}$) which makes it into a $(\mA_{\delta,\lambda+d\chi}, \mA_{\delta,\lambda})$ bimodule (resp. $(U_{\lambda+d\chi}, U_{\lambda})$ bimodule).  As either a left or right module it is locally finitely generated.

\subsection{Comments on invariants}
If $G$ is is linearly reductive then
$$\mA_{\delta,\lambda}(\chi)(U_{s \neq 0}) = \mD_X(U_{s \neq 0})^{\chi}/ \{\mu^*(\Theta)-\lambda(\Theta)-d\chi(\Theta)\}_{\Theta \in \mathfrak{g}}\mD_X(U_{s \neq 0})^{\chi}$$
and
$$U_{\lambda}(\chi) = \mD_X(X)^{\chi}/ \{\mu^*(\Theta)-\lambda(\Theta)-d\chi(\Theta)\}_{\Theta \in \mathfrak{g}}\mD_X(X)^{\chi}.$$

If $G$ is not linearly reductive the formation of quotients and $G$ invariants may not commute because the functor of $G$ invariants need not be exact.  When $G$ acts freely on $\mu^{-1}(0)^{\delta-ss}$ then the functor of $G$-invariants is an exact functor on the category of $G$-equivariant $\mO_{\mu^{-1}(0)^{\delta-ss}}$-modules.\\

\subsection{Comments on moving between characteristics}\label{betweenchars}
Let $K/\mathbb{Q}$ be a finite field extension and $S$ an affine open subset of $Spec(\overline{\mathbb{Z}}^{K})$.    Suppose that the action of $G$ on $X$ admits a model over $S$.  We would like to be able relate the existence of localization theorems for $G_{\overline{\mathbb{Q}}}$ acting on $X_{\overline{\mathbb{Q}}}$ to localization theorems for  $G_{\overline{\mathbb{F}_p}}$ acting on $X_{\overline{\mathbb{F}_p}}$.\\

For $\lambda \in \mathfrak{g}^*_S(S)$ and $\delta \in X^*(G)$ we can form the algebras from the previous section by replacing $\mD_X$ by $\mD_{X_S}$ to obtain algebras
${U_{\lambda}}_S$ and ${\mA_{\delta,\lambda}}_S$.\\

For $\overline{s}$ a geometric point of $S$,  localization holds for ${U_{\lambda}}_{S,\overline{s}} ={U_{\lambda}}_S \otimes_{\mO_S} \mO_{\overline{s}}$ if and only if the global dimension of ${U_{\lambda}}_{S,\overline{s}}$ is finite \cite{BMR} \cite{MN}.  Moreover, in situations considered in this paper, when localization holds the global dimension of ${U_{\lambda}}_{S,\overline{s}}$ is determined by the dimension of variety 
$$\left(\mu^{-1}(0) \git_\delta G \right)_{\overline{s}}$$
which is independent of $\overline{s}$.\\

The proof of the next lemma is an adaptation of \cite[3.4]{BL}.
\begin{lem} Derived  localization holds  (at $\lambda$) over $\overline{\mathbb{Q}}$ and only if it holds over $\overline{\mathbb{F}_p}$ for all but finitely many $p$.  In fact, for any $n \in \mathbb{N}$ the set 
$$T=\{ s \in S | {U_{\lambda}}_{S, s} = {U_{\lambda}}_S \otimes_{\mO_S} \mO_S/m_s \text{ has projective dimension } \leq n \text{ as a bimodule}\}$$
is  open in $S$.  \\
\end{lem}
\begin{proof} We will assume the set $T$ is not empty.   Consider that ${U_{\lambda}}_S$ is flat as an $\mO_S$-algebra as it is $\mO_S$-torsion free and $\mO_S$ is Dedekind.  Thus if $F^{\bullet}$ is a resolution of the ${U_{\lambda}}_S^e$-module ${U_{\lambda}}_S$ by finitely generated free ${U_{\lambda}}_S^e$ modules then $F^{\bullet}_s = F^{\bullet} \otimes_{\mO_S} \mO_S/m_s$ is a resolution of ${U_{\lambda}}_{S,s}$ by free ${U_{\lambda}}_{S,s}^e$ modules.  Let $K_n$ denote the $n^{th}$ syzygy module of $F^{\bullet}$. $T$ is not empty so we may choose some $s_0 \in T$ so that $(K_n)_{s_0} \oplus F_{s}^{-n}/K_n \cong F_s^{-n}$ as ${U_{\lambda}}_S^e$-modules.  Using that $F^{-n}$ is a projective ${U_{\lambda}}_S^e$-module we can lift this isomorphism to map
$$\phi: F^{-n} \rightarrow K_n \oplus F^{-n}/K_n.$$

We want to prove this map is an isomorphism on an open subset of $S$.  Both objects can be equipped with a good filtration for the graded algebra ${U_{\lambda}}_S$ and $\phi$ is a filtered map after shifting the filtration on the target.  By  moving to an open subset $V \subset S$ containing $s_0$ and shifting the filtration on the target we may assume $Gr(\phi) \neq 0$ at the stalk at $s_0$.    The morphism $\mO_S \rightarrow Gr({U_{\lambda}}_S)$ is flat and hence open.  Therefore the  image of the open set set $\{P \in Spec(Gr({U_{\lambda}}_S))_V | Gr(\phi)_P  \text{ is an isomorphism at the stalk of } P\}$  is an open set in $V$ that contains $s_0$.  To prove the latter statement, Nakayama's lemma implies that $Gr(\phi)_P$ is surjective for any $P$ over $s_0$ but $Gr(K_n \oplus F^{-n}/K_n)$ is flat over $S$ so stalk of $Kern(Gr(\phi))$ is $0$ by long exact sequence of $Tor(-,\mO_V/P)$. 

\end{proof}

\section{Generalities on derived localization for QHRs}
Recall $G$ is a reductive group acting on a smooth variety $X$ and $\delta \in X^*(G)$.  We make the following assumptions in this section.
\begin{assumptions}\label{assumptions}\

\begin{enumerate}

\item[A1.] $\lambda = d\Lambda$ for $\Lambda \in X^*(G)$.
\item[A2.]  $\mathbb{R}\Gamma(\mA_{\delta,\lambda}(\delta^m)) = \mathbb{R}Hom_{\mA_{\delta,\lambda}}(\mA_{\delta,\lambda},\mA_{\delta,\lambda}(\delta^m)) = U_{\lambda}(\delta^m)$ for all $m \geq 0$. \footnote{In characteristic $p$ this may require $p$ to be large.  For example in the case when $G=T$ is the maximal torus of  a semi-simple reductive group $G'$ acting on $X=G'/U$ one must assume that $p$ is ``very good''.}
\item[A3.] $X$ is affine of dimension $n$.
\item[A4.] $G$ acts freely on $\mu^{-1}(0)^{\delta-ss}$ and hence $\mu^{-1}(0)\git_{\delta} G$ is smooth.
\item[A5.] For the affine open sets $U_{f_i \neq 0}$ defining $\mA_{\delta,\lambda}$ we have an isomorphism of \textbf{right} $U_{\lambda}$-modules
$\mA_{\delta,\lambda}(U_{f_i \neq 0}) \cong \varinjlim_l U_{\lambda}(\delta^{lm_i})$ \footnote{This assumption is always true in characteristic $p > 0$ by the discussion in 2.5 and is true for hypertoric varieties in characteristic $0$.  In many  examples in characteristic $0$, the action of $G$ on $X$ admits a model over (an extension of) $\mathbb{Z}$ and in this case one can appeal to \ref{betweenchars}.  }

\end{enumerate}
\end{assumptions}

\subsection{The Koszul complex}

Fix $\tilde{f}_1,...,\tilde{f}_l \in \mD_X(X)$ such that
\begin{enumerate}
\item $G$ acts on $\tilde{f}_i$ by $\delta^{m_i}$ and $m_{i-1} | m_i$.
\item If $f_1,....,f_n \in \mO_{T^*X}$ denote the principal symbols of $\tilde{f}_1,...,\tilde{f}_n$ then the ideal defining the unstable locus $T^*X^{\delta-uns}$ is generated by $f_1,...,f_n$ up to radical.
\end{enumerate}

Notice that because $T^*X$ is an affine Noetherian space the ideal defining the unstable locus $T^*X^{\delta-uns}$ may be generated by a finite number of sections $\{s_{\alpha}\}$ with $\{U_{s_{\alpha} \neq 0}\}$  covering $T^*X\git_{\delta} G$ as in \ref{gitquot} .  This proves that at least one such collection $\tilde{f}_1,...,\tilde{f}_l \in \mD_X(X)$  exists.
\begin{dfc}\label{Koszulcomplex}(The $m$-shifted Koszul complex)\

Fix an integer $m$ and consider the Koszul sequence of right $\mD$-modules determined via left multiplications  by the sequence $\tilde{f}_1,...,\tilde{f}_l$.
$$0 \rightarrow \mD_X[m] \rightarrow .... \rightarrow \oplus_{j} \mD_X[m + \sum_{i \neq j} m_i] \rightarrow \mD_X[m+\sum_i m_i] \rightarrow 0$$
where $\mD_X[a]$ is the sheaf $\mD_X$ with the $G$ action twisted by $\delta^a$.\\

The cohomology sheaves of this complex are supported on the unstable locus \cite[17.13]{Eisenbud}.  Using the convention from \ref{NCloc}, the complex is exact on $T^*X^{\delta-ss}$.  Moreover, $(\{\mu^*(\Theta) -\lambda(\Theta)\}_{\Theta \in \mathfrak{g}})$ is generated by a regular sequence and $G$ acts freely on $\mu^{-1}(0)^{\delta-ss}$ so specializing along the moment map and taking invariants gives an exact complex of right $\mA_{\delta,\lambda}$-modules

$$K^{\bullet}(m): 0 \rightarrow \mA_{\delta,\lambda}(\delta^m) \rightarrow ... \rightarrow \mA_{\delta,\lambda}(\delta^{m + \sum_i m_i}) \rightarrow 0.$$  The complex $K^{\bullet}(m)$ is indexed so that the term $\mA_{\delta,\lambda}(\delta^{m + \sum_i m_i})$ is in degree zero.\\
\end{dfc}

\begin{pro}(\cite{BMR}, \cite{MN})\label{perflem} The functor of derived global sections gives an equivalence of categories
$$D^b_{gl. free}(Mod_{l.f.g}(\mA_{\delta,\lambda})) \cong D^b_{perf}(Mod_{f.g.}(U_{\lambda}))$$
where the category on the left side is the full subcategory of the derived category consisting of bounded complexes of globally free $\mA_{\delta,\lambda}$-modules.\\
\end{pro}

\begin{lem}\label{mainlem}(\textit{Alternative criteria for boundedness of localization functor.}) \\
Let $m_1,...,m_l$ be as in \ref{Koszulcomplex}.
If for all $1 \leq a < \frac{\sum_i m_i}{m_1}$  the right $\mA_{\delta,\lambda}$ module $\mA_{\delta,\lambda}(\delta^{am_1})$ is in the Karoubian triangulated subcategory of $D^b(Mod_{l.f.g}^{right}(\mA_{\delta,\lambda}))$ generated by $\mA_{\delta,\lambda}$ (as a right module over itself) then the left localization functor 
$$\mA_{\delta,\lambda} \otimes_{U_{\lambda}}^{\mathbb{L}} - : D^b(Mod_{f.g.}(U_{\lambda})) \rightarrow D^-(Mod_{l.f.g}(\mA_{\delta,\lambda}))$$
restricts to a functor $D^{[b,c]}(Mod_{f.g.}^{left}(U_{\lambda})) \rightarrow D^{[b-4n-1,c]}(Mod_{l.f.g.}^{left}(\mA_{\delta,\lambda}))$ for all $b < c$.  In particular, the localization functor is bounded.
\end{lem}
\begin{proof}\
Let $\langle \mA_{\delta,\lambda} \rangle^{right}$ denote the smallest Karoubian subcategory of $D^b(Mod_{l.f.g}^{right}(\mA_{\delta,\lambda}))$ generated by $\mA_{\delta,\lambda}$ as as right module over itself.\\

\underline{Claim 1:}  $\mA_{\delta,\lambda}(\delta^{jm_1 })$ is an object of $\langle \mA_{\delta,\lambda} \rangle^{right}$ for all $j \in \mathbb{N}$.
 \begin{lineindent}{.2in}We proceed by strong induction.  The base case is obvious so fix $j \geq 0$ and assume that $\mA_{\delta,\lambda}(\delta^{j'm_1}) \in \langle \mA_{\delta,\lambda} \rangle^{right}$ for all $0 \leq j' < j$.  If $j < \frac{\sum_i m_i}{m_1}$ then we are done so we consider the case when $j \geq \frac{\sum_i m_i}{m_1}$.\\

Let $K^{\bullet}(m)$ be the Koszul complex from $\ref{Koszulcomplex}$ and set $m=jm_1 - \sum_i m_i$.  Let $K'$ be the stupid truncation $\tau^{\leq -2} K^{\bullet}(m)$.  We  have that $\mA_{\delta,\lambda}(jm_1)=\mA_{\delta,\lambda}(\delta^{m+\sum_i m_i}) \cong Cone(K' \rightarrow \oplus_j \mA_{\delta,\lambda}(\delta^{m+\sum_{i \neq j}m_i}))$.  By inductive hypothesis, $K'$ and  $\oplus_j \mA_{\delta,\lambda}(\delta^{m+\sum_{i \neq j}m_i})$ are objects of  $\langle \mA_{\delta,\lambda} \rangle^{right}$ and therefore their cone is also.   This completes the proof of claim 1.\\
\end{lineindent}
\underline{Claim 2:}  For all $j \geq 0$ the right module $\mA_{\delta,\lambda}(\delta^{jm_1})$ is quasi-isomorphic to a complex $\mP^{\bullet}$ with $P^{i} =0$ for all $i \notin [-4n-1,0]$ ($n=dimX$) and whose terms are global direct summands of globally free finite rank right $\mA_{\delta,\lambda}$ modules.
\begin{lineindent}{.2in} Fix $j \geq 0$ and set $m=jm_1$.  By claim $1$ the right module $\mA_{\delta,\lambda}(\delta^{m})$ is perfect by virtue of being in the Karoubian subcategory generated by $\mA_{\delta,\lambda}$. This combined with assumption $A2$ implies that $U_{\lambda}(\delta^m) \cong \mathbb{R}\Gamma(\mA_{\delta,\lambda}(\delta^m)) $ is perfect so that $U_{\lambda}(\delta^m) \otimes^{\mathbb{L}}_{U_{\lambda}} \mA_{\delta,\lambda} \cong \mA_{\delta,\lambda}(\delta^m)$ by \ref{perflem}.  Thus $\mA_{\delta,\lambda}(\delta^m)$ is quasi-isomorphic to a complex $\mP^{\bullet}$ where each $\mP^{i}$ is a direct summand of a globally free finite rank right $\mA_{\delta,\lambda}$ module and $\mP^{i}=0$ for all $i \geq 1$.\\

 Let $\mK_{i}$ denote the $-i^{th}$ syzygy module of $\mP^{\bullet}$ placed in degree zero as an object of the derived category of right $\mA_{\delta,\lambda}$-modules.  If $\mP'$ is the stupid truncation $\tau^{\geq -i}\mP^{\bullet}$ then $\mK_{i}[i+1]$ is the cone of the map $\mP' \rightarrow\mA_{\delta,\lambda}(\delta^m)$.  Thus, $\mathbb{R}\Gamma(\mK_{i})[i+1]$ is the cone of $\mathbb{R}\Gamma(\mP') \rightarrow \mathbb{R}\Gamma(\mA_{\delta, \lambda}(\delta^m))$.  By assumption A2 the terms of $\mP'$ are $\Gamma$-acyclic, so the latter cone is the cone of $\Gamma(\mP') \rightarrow U_{\lambda}(\delta^m)$.  This cone is $Ker(\Gamma(\mP^{-i}) \rightarrow \Gamma(\mP^{-i+1}))[i+1]$.  Therefore, $\mathbb{R}\Gamma(\mK_{i})[i+1] \cong Ker(\Gamma(\mP^{-i}) \rightarrow \Gamma(\mP^{-i+1}))[i+1]$ which implies $ 0 = \mathbb{R}^r\Gamma(\mK_{i}) = Ext^r_{\mA_{\delta,\lambda}}(\mA_{\delta,\lambda},\mK_{i})$ for all $r \geq 1$.\\

Thus for any $r \geq 1$, $Ext^r_{\mA_{\delta,\lambda}}(\mA_{\delta,\lambda}, \mK_{i}) = 0$  implies $Ext^r_{\mA_{\delta,\lambda}}(\mP^s, \mK_{i}) = 0$ for all $s$ because $\mP^s$ is a direct summand of a globally free $\mA_{\delta,\lambda}$-module. From this we deduce the following syzygy shifting property for $\mK_i$:
$$Ext^{4n+2}_{\mA_{\delta,\lambda}}(\mA_{\delta,\lambda}(\delta^m),\mK_{i}) \cong Ext^{4n+1}_{\mA_{\delta,\lambda}}(\mK_0,\mK_{i}) \cong Ext^{4n}_{\mA_{\delta,\lambda}}(\mK_1,\mK_i) \cong ... \cong Ext^1_{\mA_{\delta,\lambda}}(\mK_{4n}, \mK_i).$$

The global dimension of $\mu^{-1}(0) \git_\delta G$ is at most $2n$ so the global dimension of $\mA_{\delta, \lambda}$ is at most $2n$. Hence the sheaf complex
$$\mathbb{R}\mHom_{\mA_{\delta,\lambda}}(\mA_{\delta,\lambda}(\delta^m),\mK_{i}) \in Ob(D^{[0,2n]}(Mod_{l.f.g.}(\mA_{\delta,\lambda})))$$ so the global complex $$\mathbb{R}Hom_{\mA_{\lambda,\delta}}(\mA_{\delta,\lambda}(\delta^m), \mK_{i}) =\mathbb{R}\Gamma(\mathbb{R}\mHom_{\mA_{\delta,\lambda}}(\mA_{\delta,\lambda}(\delta^m),\mK_{i})) \in Ob(D^{[0,4n]}(Mod_{l.f.g.}(\mA_{\delta,\lambda}))).$$  Therefore $0= Ext^{4n+2}_{\mA_{\delta,\lambda}}(\mA_{\delta,\lambda}(\delta^m),\mK_{4n+1})=Ext^1_{\mA_{\delta,\lambda}}(\mK_{4n},\mK_{4n+1})  $ yielding that $\mP^{-4n-1} = \mK_{4n} \oplus \mK_{4n+1}$.  
As $\mP^{-4n-1}$ is a direct summand of a globally free right $\mA_{\delta,\lambda}$-modules so is $\mK_{4n}$.  The smart truncation $\tau^{\geq -4n}\mP^{\bullet}$ is a complex which satisfies the conclusion of the claim.
\end{lineindent}
Recall that $\mathbb{R}\Gamma(\mA_{\delta,\lambda}(\delta^{jm_1})) \cong U_{\lambda}(\delta^{jm_1})$ from assumption A3 so by our claim each $U_{\lambda}(\delta^{jm_1})$ is quasi-isomorphic to a perfect complex $P^{\bullet}$ with $P^i = 0$ for all $i \notin [-4n-1,0]$.  Thus the right flat dimension of $U_{\lambda}(\delta^{jm_1})$ is at most $4n+1$.  By assumption A6 we have 
$$\mA_{\delta,\lambda}(U_{f_i \neq 0}) \cong \varinjlim_{l} U_{\lambda}(\delta^{ml})$$
and thus $\mA_{\delta, \lambda}(U_{f_i \neq 0})$ has right flat dimension over $U_{\lambda}$ at most $4n+1$.  The conclusion of  follows.

\end{proof}

\begin{thm}\label{mainthm}(Main Theorem) Let $\langle \mA_{\delta,\lambda} \rangle^{right}$ denote the Karoubian triangulated subcategory of $D^b(Mod^{right}_{l.f.g.}(\mA_{\delta,\lambda}))$ generated by $\mA_{\delta,\lambda}$ as a right module over itself.  The following statements are equivalent.
\begin{enumerate}
\item The global sections functor gives an equivalence of categories
$$D^b(Mod^{left}_{l.f.g.}(\mA_{\delta,\lambda})) \cong D^b(Mod_{f.g.}^{left}(U_{\lambda})).$$
\item The global sections functor gives an equivalence of categories
$$D^b(Mod^{right}_{l.f.g.}(\mA_{\delta,\lambda})) \cong D^b(Mod_{f.g.}^{right}(U_{\lambda})).$$
\item The algebra $U_{\lambda}$ has finite global dimension.
\item The algebra $U_{\lambda}$ has finite flat dimension.
\item The left localization functor $\mA_{\delta,\lambda} \otimes_{U_{\lambda}}^{\mathbb{L}} -$ is bounded.
\item $\langle \mA_{\delta,\lambda} \rangle^{right} = D^b(Mod_{f.g.}^{right}(\mA_{\delta,\lambda})).$
\item $\mA_{\delta,\lambda}(\chi) \in \langle \mA_{\delta,\lambda} \rangle^{right}$ for all $\chi \in X^*(G)$.
\item $\mA_{\delta,\lambda}(\delta^{am_1})  \in \langle \mA_{\delta,\lambda} \rangle^{right}$ for all $1 \leq a < \frac{\sum_i m_i}{m_1}$ where $f_1,...,f_l$ globally generate the ideal defining $(T^*X)^{\delta-uns}$  with each $f_i$ of weight $m_i$  (see \ref{Koszulcomplex})
\item For all $1 \leq a < \frac{\sum_i m_i}{m_1}$, where the $m_i$ are as in $8$, the natural map $$U_{\lambda}(\delta^{am_1}) \otimes^{\mathbb{L}}_{U_{\lambda}} \mA_{\delta,\lambda} \rightarrow \mA_{\delta,\lambda}(\delta^{am_1})$$ is an isomorphism.
\end{enumerate}

\end{thm}
\begin{proof}\

$(1 \Leftrightarrow 5)$ , $(2 \Leftrightarrow 3)$, and $(1 \Leftrightarrow 3)$ follow from the work of \cite{MN} in characteristic $0$ and \cite{BMR} in characteristic $p >0$.\\
$(2 \Rightarrow 9)$ Assuming global sections gives an equivalence of categories and assumption $A2$, we have that if $C$ is cone of the natural map $U_{\lambda}(\delta^{am_1}) \otimes_{U_{\lambda}}^{\mathbb{L}} \mA_{\delta,\lambda} \rightarrow \mA_{\delta,\lambda}(\delta^{am_1})$ then $\mathbb{R}\Gamma(C) \cong 0$ and so $C \cong 0$. \\
$(3 \Rightarrow 4)$ is clear.\\
$(4 \Rightarrow 5)$ is clear.\\
$(5 \Rightarrow 6)$ follows from $(5 \Rightarrow 1 \Rightarrow 3)$.  Global sections exchanges $\mA_{\delta,\lambda}$ and $U_{\lambda}$ by assumptions made on $\mA_{\delta,\lambda}$.  The finiteness of the global dimension of $U_{\lambda}$  implies that $U_{\lambda}$ Karoubian generates $D^b(Mod^{right}_{f.g.}(U_{\lambda}))$. As localization preserves Karoubian generation, $\mA_{\delta,\lambda}$ generates the derived category as proposed.\\
$(6 \Rightarrow 7)$ is clear.\\
$(7 \Rightarrow 8)$ is clear.\\
$(8 \Rightarrow 5)$ This is \ref{mainlem}.\\
$(9 \Rightarrow 8)$ We can compute the tensor product by taking a (possibly unbounded) resolution of $U_{\lambda}(\delta^m)$ by finite rank free right $U_{\lambda}$ modules, thus $\mA_{\delta,\lambda}(\delta^m)$ has an (possibly unbounded) resolution by globally free finite rank right $\mA_{\delta,\lambda}$ modules, $\mF^{\bullet}$.  The category $Mod_{l.f.g.}(\mA_{\delta,\lambda})$ has finite homological dimension so by \cite[3.4.5, version 1]{BMR} $\mA_{\delta,\lambda}(\delta^m)$ is a direct summand of a stupid truncation of $\mF^{\bullet}$.  Thus, $\mA_{\delta,\lambda}(\delta^m) \in \langle \mA_{\delta,\lambda} \rangle^{right}$.

\end{proof}

\begin{rem}Condition $9$ can be considered analogous to the Borel-Weil-Bott theorem for line bundles on the flag variety.
\end{rem}

\section{Examples}
\subsection{Transfer bimodules and Morita contexts}
We will now briefly revisit some standard results and techniques in noncommutative algebra and Morita theory.  A good resource for the statements in the affine cases is \cite{McR}.\\

\begin{df}(Transfer bimodules) For $\Lambda_1, \Lambda_2 \in X^*(G)$ and $\lambda_i = d\Lambda_i$ define the transfer bimodules
$$\mB^{\Lambda_1, \Lambda_2}_{\delta}(U_{s \neq 0})= \left(\mD_X/ \{\mu^*(\Theta)-\lambda_1(\Theta)\}_{\Theta \in \mathfrak{g}}\mD_X(U_{s \neq 0})\right)^{\Lambda_1 \Lambda_2^{-1}}$$
and
$$B^{\Lambda_1, \Lambda_2} =  \left(\mD_X(X)/ \{\mu^*(\Theta)-\lambda_1(\Theta)\}_{\Theta \in \mathfrak{g}}\mD_X(X)\right)^{\Lambda_1 \Lambda_2^{-1}}.$$
\end{df}

$\mB_{\delta}^{\Lambda_1,\Lambda_2}$ is an ($\mA_{ \delta,\lambda_1},\mA_{\delta, \lambda_2 }$) bimodule and  $B^{\Lambda_1,\Lambda_2}$ is a $(U_{\lambda_1}, U_{\lambda_2})$ bimodule.\\

Notice that if $\lambda = d\Lambda$ for $\Lambda \in X^*(G)$ then
$$\mA_{\delta,\lambda}(\chi) = \mB_{\delta}^{\Lambda+\chi,\Lambda} \text{ and } U_{\lambda}(\chi) = B^{\Lambda+\chi, \Lambda}.$$

There are natural multiplication maps 
$$m_1: B^{\Lambda_1,\Lambda_2} \otimes_{U_{\lambda_2}} B^{\Lambda_2, \Lambda_1} \rightarrow U_{\lambda_1} \text{ and } m_2: B^{\Lambda_2,\Lambda_1} \otimes_{U_{\lambda_1}} B^{\Lambda_1, \Lambda_2} \rightarrow U_{\lambda_2}$$

which make the sextuplet $(U_{\lambda_1},U_{\lambda_2}, B^{\Lambda_1,\Lambda_2}, B^{\Lambda_2,\Lambda_1},  m_1, m_2)$  a Morita context. \\

We can use these Morita contexts to induce a relation preordering $X^*(G)$.
\begin{df}(The surjection relation)\footnote{This relation is much harder to satisfy when the characteristic of $\boldk$ is $p > 0$.  For example, consider the diagonal action of $\mathbb{G}_m$ on $\mathbb{A}^2$.  In characteristic $0$ this equivalence relation determines at most three equivalence classes but determines an infinite number of equivalence classes in characteristic $p >0$.}\

 For $\Lambda_1, \Lambda_2 \in X^*(G)$  and $\lambda_i = d\Lambda_i$ define $\Lambda_2 \geq \Lambda_1$ if and only if $B^{\Lambda_1, \Lambda_2} \otimes_{U_{\lambda_2}} B^{\Lambda_2, \Lambda_1} \rightarrow U_{\lambda_1}$ is surjective.
\end{df}

\begin{pro} $\geq$ is a preorder on $X^*(G)$.
\end{pro}

\begin{pro} If $\Lambda_2 \geq \Lambda_1$ then $B^{\Lambda_1,\Lambda_2}$ is projective as a right $U_{\lambda_2}$-module and the functor $ B^{\Lambda_1,\Lambda_2} \otimes_{U_{\lambda_2}} -$  is faithful.  If $\Lambda_1 \geq \Lambda_2$ and $\Lambda_2 \geq \Lambda_1$ then $U_{\lambda_1}$ and $U_{\lambda_2}$ are Morita equivalent.
\end{pro}
\begin{proof}
This is a standard result in noncommutative ring theory, see \cite{McR} for further details.
\end{proof}

\begin{pro}\label{Moritapro} If $\Lambda_1 \geq \Lambda_2 \geq \Lambda_3$ and $\Lambda_3 \geq \Lambda_1$ then $B^{\Lambda_1,\Lambda_2} \otimes_{U_{\lambda_3}} B^{\Lambda_2, \Lambda_3} \cong B^{\Lambda_1, \Lambda_3}$.
\end{pro}
\begin{proof} As $\Lambda_1 \geq \Lambda_2 \geq \Lambda_3$ multiplication determines a surjective map
$$B^{\Lambda_3,\Lambda_2} \otimes_{U_{\lambda_2}} B^{\Lambda_2, \Lambda_1} \otimes_{U_{\lambda_1}} B^{\Lambda_1, \Lambda_2} \otimes_{U_{\lambda_2}} B^{\Lambda_2, \Lambda_3} \rightarrow U_{\lambda_3}$$
which factors through $B^{\Lambda_3,\Lambda_1} \otimes_{U_{\lambda_1}} B^{\Lambda_1,\Lambda_2} \otimes_{U_{\lambda_2}} B^{\Lambda_2,\Lambda_3}$ via $a \otimes b \otimes c \otimes d \mapsto ab \otimes c \otimes d$.  By \cite[5.4]{McR} $B^{\Lambda_1,\Lambda_2} \otimes_{U_{\lambda_2}} B^{\Lambda_2,\Lambda_3}$ is bimodule isomorphic to the $U_{\lambda_1}$ dual of $B^{\Lambda_3,\Lambda_1}$.  $\Lambda_3 \geq \Lambda_1$ so this dual is precisely $B^{\Lambda_1,\Lambda_3}$.
\end{proof}

\subsection{Example: Localization from a finite number of Morita equivalences}
In characteristic $0$, it is known that having an infinite number of Morita equivalences can be used to prove (underived) localization theorems.  For a  discussion about this technique see \cite{BPW}.  In this section, we will show that a finite number of Morita equivalences is sufficient to imply derived localization holds.

\begin{exa}\label{enoughmorita}(Derived Localization from a finite number of Morita equivalences)\\

Suppose  that $\Lambda\delta^{am_1} \leq \Lambda \leq \Lambda \delta^{am_1}$ for all $1 \leq a \leq \frac{\sum_i m_i}{m_1}$ (in the notation of \ref{Koszulcomplex}) then derived global sections gives an equivalence of categories
$$D^b(Mod_{l.f.g.}(\mA_{\delta,\lambda})) \cong D^b(Mod_{f.g.}(U_{\lambda})).$$
\begin{proof}
Note that $\mA_{\delta,\lambda} = \mA_{\delta^{m_1},\lambda}$ and the same follows for the transfer bi-modules, so without loss of generality we may assume $m_1=1$.
We will prove the following four maps are isomorphisms.  The equivalence of categories follows from realizing I4 is equivalent to condition $9$ of \ref{mainthm}.
\begin{enumerate}
\item[I1.] $U_{\lambda}(\delta^a) \otimes_{U_{\lambda}} \mB_{\delta}^{\Lambda,\Lambda\delta^a} \rightarrow \mA_{\delta, \lambda+ad\delta}$ is an isomorphism for all $1 \leq a \leq \sum_i m_i$.
\item[I2.] $\mA_{\delta,\lambda}(\delta^a) \otimes_{\mA_{\delta,\lambda}} \mB_{\delta}^{\Lambda,\Lambda\delta^a} \rightarrow \mA_{\delta, \lambda+ad\delta}$ is an isomorphism for all $1 \leq a \leq \sum_i m_i$.
\item[I3.] $\mB_{\delta}^{\Lambda,\Lambda\delta^a} \otimes_{\mA_{\delta,\lambda+ad\delta}} \mA_{\delta,\lambda}(\delta^a) \rightarrow \mA_{\delta,\lambda}$ is an isomorphism for all $1 \leq a \leq \sum_i m_i$ .
\item[I4.] $U_{\lambda}(\delta^a) \otimes^{\mathbb{L}}_{U_{\lambda}} \mA_{\delta,\lambda} \rightarrow \mA_{\delta,\lambda}(\delta^a)$ is an isomorphism for all $1 \leq a \leq \sum_i m_i$.
\end{enumerate}
From the relation $\Lambda\delta^a \geq \Lambda \geq \Lambda \delta^a$ the multiplication maps 
$$U_{\lambda}(\delta^a) \otimes_{U_{\lambda}} B^{\Lambda,\Lambda \delta^a} \rightarrow U_{\lambda+ad\delta}$$
and 
$$B^{\Lambda,\Lambda \delta^a} \otimes_{U_{\lambda+ad\delta}} U_{\lambda}(\delta^a) \rightarrow U_{\lambda}$$
are surjective.\\ 

This gives that the maps in I1, I2, and I3 are surjective. On each $U_{f_i \neq 0}$ the maps in I2 and I3 are part of a Morita context.  It follows from results on Morita contexts (\cite[5.4]{McR}) that the maps I2 and I3 are isomorphisms because they are surjective.  $B^{\Lambda,\Lambda \delta^a}$ is part of a Morita equivalence so it is faithfully flat as a right $U_{\lambda+ad\delta}$-module.  Tensoring  I1 with $B^{\Lambda,\Lambda \delta^a}$ gives the factorization 
$$\mB_{\delta}^{\Lambda,\Lambda \delta^a} \cong B^{\Lambda,\Lambda \delta^a} \otimes_{U_{\lambda+ad\delta}} B^{\Lambda\delta^a,\Lambda} \otimes_{U_{\lambda}} \mB_{\delta}^{\Lambda,\Lambda \delta^a} \stackrel{B^{\Lambda,\Lambda \delta^a} \otimes (I1)}{\twoheadrightarrow}  B^{\Lambda, \Lambda \delta^a} \otimes_{U_{\lambda+ad\delta}} \mA_{\lambda+ad\delta} \rightarrow \mB_{\delta}^{\Lambda,\Lambda\delta^a}$$ of the identity map.  This proves that $I1$ is also injective and hence an isomorphism.\\

The arrow in I4 factors as
$$U_{\lambda}(\delta^a) \otimes_{U_{\lambda}} \mA_{\delta,\lambda} \cong U_{\lambda}(\delta^a) \otimes_{U_{\lambda}} \mB_{\delta}^{\Lambda,\Lambda\delta^a} \otimes_{\mA_{\delta, \lambda+ad\delta}} \mA_{\delta,\lambda}(\delta^a) \longrightarrow \mA_{\delta, \lambda+ad\delta} \otimes_{\mA_{\lambda+ad\delta}} \mA_{\delta,\lambda}(\delta^a)\cong  \mA_{\delta,\lambda}(\delta^a) $$
where the middle diagram is $(I1) \otimes \mA_{\delta,\lambda}(\delta^a)$.  This proves that $I4$ is an isomorphism.

\end{proof}

\end{exa}

\subsection{Example: $T^*\mathbb{P}^{n-1}$ and generalizations  in characteristic $p$}
Before proceeding to the general case of hypertoric varieties, we will first start with an example where the combinatorial data is easy to understand.  No results from this section will be referenced by the next section.  This subsection generalizes the results of \cite{VdB} to characteristic $p  >0 $. \\

Let $X=\mathbb{A}^n = Spec(\boldk[x_1,...,x_n])$, $\mD_X(X) = \boldk \langle x_1,...,x_n, \partial_1,...,\partial_n \rangle$, $s \in [1,n]$ an integer (the signature) and $G=\mathbb{G}_m$ acting on $X$ via $$t(x_1,...,x_s,x_{s+1},...,x_n) = (tx_1,...,tx_s,t^{-1}x_{s+1},...,t^{-1}x_n).$$
 on the points of the space and corresponding right regular action on functions.
Choose $\delta \in X^*(G)$ to be the character $t \mapsto t$.  For example, in the case when $s=n$ we obtain
$$\mu^{-1}(0) \git_{\delta} G \cong T^*\mathbb{P}^{n-1}.$$

For $\Lambda \in X^*(G)$ and $\lambda=d\Lambda$ the ideal $(\{\mu^*(\Theta)-\lambda(\Theta)\}_{\Theta \in \mathfrak{g}})$ is generated by the single element
$$\sum_{i=1}^s x_i\partial_i + \sum_{i=s+1}^n -x_i \partial_i - \lambda.$$

For assumption A5, take $f_1 = x_1.,..,f_s=x_s,f_{s+1}=\xi_{s+1},...,f_{n}=\xi_n$ then $m_i = 1$ for all $i$.\\

Clearly, $$B^{\Lambda \delta,\Lambda} \otimes B^{\Lambda, \Lambda\delta} \ni \sum_{i=1}^{s} x_i \otimes \partial_i + \sum_{i=s+1}^n -\partial_i \otimes x_i \mapsto \sum_{i=1}^s x_i\partial_i + \sum_{i=s+1}^n- x_i \partial_i - n + s \in U_{\lambda+d\delta}$$

and

$$B^{\Lambda,\Lambda\delta} \otimes B^{\Lambda\delta, \Lambda} \ni \sum_{i=1}^{s} \partial_i \otimes x_i + \sum_{i=s+1}^n -x_i \otimes \partial_i \mapsto \sum_{i=1}^s x_i\partial_i + \sum_{i=s+1}^n -x_i \partial_i + s \in U_{\lambda}.$$

We see that if $\lambda + s$ and $\lambda+s -n$ are both invertible modulo $p$ then $\Lambda\delta \leq \Lambda \leq \Lambda\delta$. \\

Therefore, whenever $\lambda \notin \{-s+n-j, -j-s\}_{0 \leq j < n}$ we have that $\Lambda\delta^a \leq \Lambda \leq \Lambda\delta^a$ for all $1 \leq a \leq n$.  By \ref{enoughmorita} for any $\Lambda$ with $\lambda \notin \{-s+n-j, -j-s\}_{0 \leq j < n}$ there is an equivalence of categories
$$D^b(Mod_{f.g.}(U_{\lambda})) \cong D^b(Mod_{l.f.g.}(\mA_{\delta,\lambda})).$$

If $p > 2n$ then there are at least $p-2n$ such $\Lambda \in X^*(G)$ where derived equivalence holds at $d\Lambda$.

\subsection{Example: Hypertoric varieties in characteristic $p >0$}

In this subsection we will analyze the case when $X=\mathbb{A}^n$ and $G \subset T=\mathbb{G}_m^n$ is a subtorus acting on $X$ via the natural action of $T$ on $X$. To avoid the trivial cases, we will assume $0 < dim(G) < n$.  The GIT quotients $\mu^{-1}(0) \git_{\delta} G$ arising in this manner are called hypertoric varieties.  The action of $G$ on $X$ is entirely specified by the induced morphism of groups $\pi: X_*(T) \rightarrow X_*(T/G)$ which independent of the base field $\boldk$.  For simplicity we will further assume that the induced morphism is unimodular.\footnote{Here we mean any set of linearly independent columns is part of a $\mathbb{Z}$-basis for the codomain.} This setting was studied in characteristic zero in \cite{MVdB}.\\

\subsubsection{Notation.}
\begin{itemize}
\item $\langle \hspace{5pt}, \hspace{5pt} \rangle$ denotes the natural pairing between a space and its dual.
\item  $\widetilde{T}$ denotes the group $\mathbb{G}_m^{2n}$ acting naturally  on $\mathbb{A}^{2n} \cong T^*X$.  There are obvious inclusions $G \subset T \subset \tilde{T}$.
\item  $\{\tilde{e}_i\}_{i=1}^{2n}$ is the standard basis for $X_*(\tilde{T})$.
\item For $I \subset \{1,...,n\}$ define $\mathbb{G}_n^{|I|}  \cong T_I \subset T \subset \tilde{T}$ as the subtorus generated by the one parameter subgroups $\tilde{e}_i-\tilde{e}_{n+i}$ for all $i \notin I$.  
\item For any $I$ such that $X_*(G) \cap X_*(T_I)$ is rank $1$ we define the wall
$$W_I = \{\chi \in X^*(G) | \langle \chi, X_*(G) \cap X_*(T_I) \rangle = 0\}.$$
\end{itemize}

\subsubsection{Definition of smooth parameter.}
We will say that $\delta \in X^*(G)$ is a smooth parameter if $G$ acts freely on $\mu^{-1}(0)$.

\subsubsection{Construction of the Koszul sequence.}
In order to utilize \ref{mainthm} we first need to fix a sequence generating the unstable locus $T^*X^{\delta-uns}$ up to radical.\\

 Consider the polyhedron
$$P_{\delta} = \{ \tilde{\chi} \in X^*(\tilde{T})_{\mathbb{R}} \text{: } \tilde{\chi}|_G = \delta \text{ and } \langle \tilde{\chi}, \tilde{e_i} \rangle \geq 0 \text{ for all } 1 \leq i \leq 2n\}.$$

 For any integral character $\tilde{\chi} \in P_{\delta}$ define
 $$\tilde{f}_{\tilde{\chi}} = \prod_{1 \leq i \leq n}x^{\langle \tilde{\chi}, \tilde{e_i} \rangle}_i \prod_{1 \leq i \leq n} \partial_i^{\langle \tilde{\chi}, \tilde{e}_{i+n} \rangle}.$$

The vertices of $P_{\delta}$ are integral because the map $X_*(T) \rightarrow X_*(T/G)$ is unimodular.  Label them $v_1(\delta),...,v_{s_{\delta}}(\delta)$.  The regular functions $f_{v_1(\delta)},...,f_{v_s(\delta)}$ (the principal symbols of $\tilde{f}_{v_1(\delta)},....,\tilde{f}_{v_s(\delta)})$ generate the ideal defining the unstable locus $T^*X^{\delta-uns}$ \cite[2.6]{HaSturm}.  \\

\textbf{The defining sequence.} For any $\delta \in X^*(G)$ that is a smooth parameter we will choose $\tilde{f}_{v_1(\delta)},....,\tilde{f}_{v_{s_{\delta}}(\delta)}$ for the data required by \ref{Koszulcomplex} for \ref{mainthm}.  Notice that in this case $m_i=1$ for all $ 1 \leq i \leq s_{\delta}$. \\

\subsubsection{Explicit lower bounds on $p$.}

For each vertex $v_i(\delta)$  we define $N_i(\delta)=\max_j | \langle v_i(\delta), \tilde{e}_j \rangle |$ and $N(\delta) = max_i N_i(\delta)$.
Further more we define $N = \min_{\delta} N(\delta)$ as $\delta \in X^*(G)$ ranges over those $\delta$ not contained in any $W_I$ for $rank(X_*(G) \cap X_*(T_I))=1$  nor in any $Span_{\mathbb{Z}}\{\tilde{e_i}|_G \}_{i \in I}$ with $|I| < dim \mathfrak{g}$.  
We will prove the following in \ref{pproofs}.

\begin{enumerate}
\item If $p > \binom{n}{dim\mathfrak{g}-1}(2N)^{n-dim\mathfrak{g}+1}$ then there is a smooth parameter $\delta \in X^*(G)$ and $\Lambda \in X^*(G)$ such that $\Lambda \leq \Lambda \delta$.
\item If $p > 2\binom{n}{dim\mathfrak{g}}\binom{n}{dim\mathfrak{g}-1}\left(2\binom{n}{dim\mathfrak{g}-1}N+1\right)^{n-dim\mathfrak{g}+1}$ then there is a smooth parameter $\delta \in X^*(G)$ and  $\Lambda \in X^*(G)$ such that localization holds at $\lambda= d\Lambda$.
\end{enumerate}

\begin{rem}\label{easybndrmk}It is easy to find an upper bound for these estimates because  $N$ is defined as a minimal value. 
\end{rem}
\subsubsection{Confirmation of the axioms from \ref{assumptions} when $\delta$ is a smooth parameter.}

We will show that for any choice of smooth parameter $\delta \in X^*(G)$ all of the axioms from \ref{assumptions} are satisfied.\\

Suppose $\delta$ is a smooth parameter and $\Lambda \in X^*(G)$.   Only A4 of \ref{assumptions} remains to be confirmed.  We need to check that the cohomology sheaves $H^i(\mu^{-0}\git_{\delta}G, \mA_{\delta, d\Lambda}(\delta^m))$ vanish for $i > 0$ and $m \geq 0$  As explained in \cite[Sec 4]{Sta},  the ideal $\{\mu^*(\theta)-\lambda(\theta)\}_{\theta \in \mathfrak{g}^*}$ is generated by a regular sequence it so it is enough to check that $H^i(T^*X\git_{\delta} G, \mD_{T^*X^{\delta-ss}}^{\delta^m}) = 0$ for all $ i > 0$ and $m \geq 0$.  Recalling that $\mD_X$ is naturally filtered with a PBW filtration such that $Gr(\mD_{T^*X^{\delta-ss}}) = \mO_{T^*X\git_{\delta}G}(\delta^m)$ the usual spectral sequence arguments imply that it suffices to show the higher cohomology sheaves of $\mO_{T^*X\git_{\delta}G}(\delta^m)$ vanish.  It was shown in \cite[Sec 4]{Sta} that the variety $T^*X\git_{\delta}G$ is Frobenius split.   A quick argument  (see \cite[Intro]{MR}) shows that the ample line bundle $\mO_{T^*X\git_{\delta}G}(\delta^m)$  on the Frobenius split variety $T^*X \git_{\delta} G$ has vanishing higher cohomology groups.\\

\subsubsection{Proofs of the explicit bounds on $p$.}\label{pproofs}

\begin{pro} If $p > \binom{n}{dim\mathfrak{g}-1}(2N)^{n-dim\mathfrak{g}+1}$ then there exists a smooth parameter $\delta \in X^*(G)$ and $\Lambda \in X^*(G)$ such that $\Lambda \leq \Lambda \delta$.
\end{pro}
\begin{proof}
Suppose that $p > \binom{n}{dim\mathfrak{g}-1}(2N)^{n-dim\mathfrak{g}+1}$.
Let $\delta$ be such that $N(\delta) = N$ where $\delta$ is not in any $W_I$ nor $Span_{\mathbb{Z}}\{\tilde{e}_i\}_I$ for $|I| < dim \mathfrak{g}$.  $\delta$ is a smooth parameter because $\delta \notin \cup_I W_I$ \cite[Sec 5]{Sta}.  Write $v_i$ for $v_i(\delta)$ (recall $v_i(\delta)$ are the vertices of $P_{\delta}$) and consider the elements
$$\tilde{f}_{v_i} = \prod_{1 \leq j \leq n}x_j^{\langle v_i, \tilde{e_j} \rangle} \prod_{1 \leq j \leq n} \partial_j^{\langle v_i, \tilde{e}_{j+n} \rangle} \in \mD_X(X)^{\delta}$$
$$\tilde{g}_{v_i} = \prod_{1 \leq j \leq n}\partial_j^{\langle v_i, \tilde{e_j} \rangle} \prod_{1 \leq j \leq n}x_j^{\langle v_i, \tilde{e}_{j+n} \rangle} \in \mD_X(X)^{\delta^{-1}}.$$

The product $\tilde{g}_{v_i}\tilde{f}_{v_i} \in \mD_X(X)$ is in the commutative $\boldk$-subalgebra generated by $E_i = x_i \partial_i$. We naturally identify $\boldk[E_1,...,E_n]$ with the ring of regular functions on $\mathfrak{t}^*$ and the inclusion into $\mD_X(X)$ is given by the moment map for the action of $T$ on $X$.  Utilizing the identities $x^m\partial^m = \prod_{k=0}^{m-1}(x\partial-k)$ and $\partial^m x^m = \prod_{k=1}^{m}(x\partial+k)$, we can express these products in terms of the generators of $\boldk[E_1,...,E_n]$ as
$$\tilde{f}_{v_i}\tilde{g}_{v_i} =  \prod_{1 \leq j \leq n}\prod_{k=0}^{\langle v_i, \tilde{e}_j \rangle -1} (E_j-k) \cdot \prod_{1 \leq j \leq n}\prod_{k=1}^{\langle v_i, \tilde{e}_{n+j} \rangle }( E_j+k)$$
$$\tilde{g}_{v_i}\tilde{f}_{v_i}  =  \prod_{1 \leq j \leq n}\prod_{k=1}^{\langle v_i, \tilde{e}_{j} \rangle }( E_j+k) \cdot \prod_{1 \leq j \leq n}\prod_{k=0}^{\langle v_i, \tilde{e}_{n+j} \rangle -1} (E_j-k)  $$

Any common root of $\tilde{g}_{v_i}\tilde{f}_{v_i}$ for $1 \leq i \leq s$ is in $\mathfrak{t}^*(\mathbb{F}_p)$.    As $\delta \notin Span_{\mathbb{Z}}\{\tilde{e}_i\}_I$ for $|I| < dim \mathfrak{g}$  

we have that the set $S_i = \{ j \in \{1,...,n\} | \langle v_i, \tilde{e}_j-\tilde{e}_{n+j} \rangle \neq 0\}$ has exactly $dim\mathfrak{g}$ elements and every set of $dim\mathfrak{g}$ elements shows up as some $S_i$.   Furthermore, the set of common roots of $\tilde{g}_{v_i}\tilde{f}_{v_i}$ for $1 \leq i \leq s$ is contained in the set
$$\bigcap_i \bigcup_{j \in S_i}  \left\{ \sum_{l=0}^n a_l d(\tilde{e}_l^{\vee}-\tilde{e}_{n+l}^{\vee}) | a_l \in \mathbb{Z}, 0 \leq a_l < p \text{ and } -N \leq a_j < N\right\} \subset \mathfrak{t}^*(\mathbb{F}_p).$$

This set has size at most $\binom{n}{dim\mathfrak{g}-1}(2N)^{n-dim\mathfrak{g}+1}p^{dim\mathfrak{g}-1}$.  We conclude that there are at most $\binom{n}{dim\mathfrak{g}-1}(2N)^{n-dim\mathfrak{g}+1}p^{dim\mathfrak{g}-1}$ common roots of  $\{\tilde{g}_{v_i}\tilde{f}_{v_i}\}_{i=1}^s$.\\

For any $\xi \in\mathfrak{t}^*(\mathbb{F}_p)$ let $\mu^{\vee}_{\xi} \in \mathfrak{g}^*(\mathbb{F}_p)$ denote the functional defined by
$$\theta \mapsto \mu^*(\theta)(\xi)$$
where $\mu^*(\theta) \in \boldk[E_1,...,E_n] \subset \mD_X(X)$ is considered as a regular function on $\mathfrak{t}^*$.

$p > \binom{n}{dim\mathfrak{g}-1}(2N)^{n-dim\mathfrak{g}+1}$ so the set 
$$g^*(\mathbb{F}_p)\setminus \{\mu^{\vee}_{\xi}| \xi \text{ is a common root of } \{\tilde{g}_{v_i}\tilde{f}_{v_i}\}_{i=1}^s\}$$
is non empty.  Choose any $\Lambda$ in $X^*(G)$ such that $d\Lambda \notin \{\mu^{\vee}_{\xi}| \xi \text{ is a common root of } \{\tilde{g}_{v_i}\tilde{f}_{v_i}\}_{i=1}^s\}$. For any common root $\xi$ of $\{\tilde{f}_{v_i}\tilde{g}_{v_i}\}_{i=1}^s$ there exists $\theta_\xi \in \mathfrak{g}^*$ such that  $\lambda(\theta_\xi) \neq \mu^{\vee}_{\xi}(\theta_\xi) = \mu^*(\theta_\xi)(\xi)$. $\boldk$ is algebraically closed so the ideal generated by the set $\{\mu^*(\theta)-\lambda(\theta)\}_{\theta \in \mathfrak{g}^*} \cup \{\tilde{g}_{v_i}\tilde{f}_{v_i}\}_{i=1}^s \subset \boldk[E_1,...,E_n]$ is the unit ideal. The two-sided $U_{\lambda}$ ideal generated by the images of $\tilde{g}_{v_i} \otimes \tilde{f}_{v_i} \in B^{\Lambda, \Lambda\delta} \otimes B^{\Lambda\delta,\Lambda}$ $(1 \leq i \leq s)$ in $U_{\lambda}$ contains $1$.  It follows that that $\Lambda \leq \Lambda\delta$.
\end{proof}

\begin{thm}\label{mainapp} If $p > 2\binom{n}{dim\mathfrak{g}}\binom{n}{dim\mathfrak{g}-1}\left(2\binom{n}{dim\mathfrak{g}-1}N+1\right)^{n-dim\mathfrak{g}+1} $ then there exists a smooth parameter $\delta \in X^*(G)$ such that the set of integral weights $\lambda \in \mathfrak{g}^*(\mathbb{F}_p)$ for which derived localization holds is non-empty.
\end{thm}
\begin{proof}
We proceed in a manner similar to the previous proof. Let $\delta$ be such that $N(\delta) = N$ where $\delta$ is not in any $W_I$ nor $Span_{\mathbb{Z}}\{\tilde{e}_i\}$ for $|I| < dim \mathfrak{g}$.  Write $v_i$ for $v_i(\delta)$ and consider the elements
$$\tilde{f^a}_{v_i} = \prod_{1 \leq j \leq n}x_j^{a\langle v_i, \tilde{e_j} \rangle} \prod_{1 \leq j \leq n} \partial_j^{a\langle v_i, \tilde{e}_{j+n} \rangle} \in \mD_X(X)^{\delta}$$
$$\tilde{g^a}_{v_i} = \prod_{1 \leq j \leq n}\partial_j^{a\langle v_i, \tilde{e_j} \rangle} \prod_{1 \leq j \leq n}x_j^{a\langle v_i, \tilde{e}_{j+n} \rangle} \in \mD_X(X)^{\delta^{-1}}$$
for any $1 \leq a \leq \binom{n}{dim\mathfrak{g}}$.\\

The products $\tilde{f}_{v_i}\tilde{g}_{v_i}$ and $\tilde{g}_{v_i}\tilde{f}_{v_i} \in \mD_X(X)$ are in the commutative $\boldk$-subalgebra generated by $E_i = x_i \partial_i$. We again use the identities $x^m\partial^m = \prod_{k=0}^{m-1}(x\partial-k)$ and $\partial^m x^m = \prod_{k=1}^{m}(x\partial+k)$  to express these product in terms of the generators $E_i$ as

$$\tilde{f^a}_{v_i}\tilde{g^a}_{v_i} =  \prod_{1 \leq j \leq n}\prod_{k=0}^{a\langle v_i, \tilde{e}_j \rangle -1} (E_j-k) \cdot \prod_{1 \leq j \leq n}\prod_{k=1}^{a\langle v_i, \tilde{e}_{n+j} \rangle }( E_j+k)$$
$$\tilde{g^a}_{v_i}\tilde{f^a}_{v_i}  =  \prod_{1 \leq j \leq n}\prod_{k=1}^{a\langle v_i, \tilde{e}_{j} \rangle }( E_j+k) \cdot \prod_{1 \leq j \leq n}\prod_{k=0}^{a\langle v_i, \tilde{e}_{n+j} \rangle -1} (E_j-k)  $$

For fixed $a \leq \binom{n}{dim\mathfrak{g}}$ the set of common roots of  $\{\tilde{f^a}_{v_i}\tilde{g^a}_{v_i}\}_{i=1}^{s}$ is contained in the set
$$\bigcap_i \bigcup_{j \in S_i} = \left\{ \sum_{l}a_l d(\tilde{e}_l^{\vee}-\tilde{e}_{n+l}^{\vee}) | a_l \in \mathbb{Z}, 0 \leq a_l < p \text{ and }- \binom{n}{dim\mathfrak{g}-1}N \leq a_j \leq \binom{n}{dim\mathfrak{g}-1}N\right\}.$$
This set also contains the set of all common roots of  $\{\tilde{g^a}_{v_i}\tilde{f^a}_{v_i}\}_{i=1}^{s}$ for fixed $a \leq \binom{n}{dim\mathfrak{g}}$.  This set has at most $\binom{n}{dim\mathfrak{g}-1}\left(2\binom{n}{dim\mathfrak{g}-1}N+1\right)^{n-dim\mathfrak{g}+1}p^{dim\mathfrak{g}-1}$ elements.\\

For $\xi \in \mathfrak{t}^*$ we let $\mu^{\vee}_{\xi}$ be as in the proof of the previous proposition.  The set
$$ \cup_{1 \leq a \leq s} \{ \mu_{\xi}^{\vee}| \xi \text{ is a c. root of } \{\tilde{g^a}_{v_i}\tilde{f^a}_{v_i}\}_{i=1}^{s}\} \bigcup  \cup_{1 \leq a \leq s} \{ \mu_{\xi}^{\vee} - d\delta^a | \xi \text{is a c. root of } \{\tilde{f^a}_{v_i}\tilde{g^a}_{v_i }\}_{i=1}^{s}\}$$
has at most $2\binom{n}{dim\mathfrak{g}}\binom{n}{dim\mathfrak{g}-1}\left(2\binom{n}{dim\mathfrak{g}-1}N+1\right)^{n-dim\mathfrak{g}+1}p^{dim\mathfrak{g}-1}$ elements. Therefore as long as $p > 2\binom{n}{dim\mathfrak{g}}\binom{n}{dim\mathfrak{g}-1}\left(2\binom{n}{dim\mathfrak{g}-1}N+1\right)^{n-dim\mathfrak{g}+1}$ we can choose $\Lambda$  such that $d\Lambda$ is not in this set.\\

Referencing the proof of the previous proposition, we have $\Lambda \leq \Lambda \delta^a$ for all $1 \leq a \leq \binom{n}{dim \mathfrak{g}}$.  Moreover, similar to the previous proof for fixed $a$ the images of
$\tilde{f^a}_{v_i} \otimes \tilde{g^a}_{v_i } \in B^{\Lambda \delta^a, \Lambda} \otimes B^{\Lambda, \Lambda \delta^a}$ $(1 \leq i \leq s)$ generate the trivial ideal in $U_{\lambda + d\delta^a}$ because for every common root $\xi$ of $\{\tilde{f^a}_{v_i}\tilde{g^a}_{v_i }\}_{i=1}^{s}$  we have that $\mu^{\vee}_{\xi}(\theta_\xi) \neq \lambda(\theta_\xi) + d\delta^a(\theta_\xi)$ for some $\theta_\xi \in \mathfrak{g}$.  This yields that $\Lambda \delta^a \leq \Lambda \leq \Lambda \delta^a$ for all $1 \leq a \leq s$ and hence localization holds at $\Lambda$ by \ref{enoughmorita}.

\end{proof}

\def\cprime{$'$} \def\cprime{$'$}

E-mail: tstadnik@msri.org
\end{document}